\documentclass[12pt]{amsart}
\usepackage[utf8]{inputenc}
\usepackage{amssymb}
\usepackage{amsmath}
\usepackage{amsthm}
\usepackage{graphicx}
\usepackage{amscd}
\usepackage{epic}
\usepackage[alphabetic]{amsrefs}
\usepackage{xypic}
\usepackage[colorlinks, citecolor=magenta]{hyperref}
\usepackage[left=1.5in,right=1.5in,top=1.5in,bottom=1.5in]{geometry}
\usepackage{soul}
\usepackage[pagewise]{lineno}
\usepackage{comment}
\usepackage{enumitem}
\usepackage{mathrsfs}
\usepackage{tikz-cd}

\theoremstyle{plain}
\newtheorem{theorem}{Theorem}[section]

\newtheorem{lemma}[theorem]{Lemma}

\newtheorem{definition-lemma}[theorem]{Definition-Lemma}
\theoremstyle{remark}
\newtheorem{remark}[theorem]{Remark}

\theoremstyle{definition}

\newcommand{\Pic}[0]{\operatorname{Pic}}

\def\NE{\overline{\operatorname{NE}}}

\def\NA{\overline{\operatorname{NA}}}

\def\>{\geq}
\newcommand{\mbQ}{\mathbb{Q}}
\newcommand{\mbR}{\mathbb{R}}

\newcommand{\mult}{{\rm mult}}

\newcommand{\mbC}{\mathbb{C}}
\newcommand{\mbZ}{\mathbb{Z}}

\def\>{\geq}

\def\mult{\operatorname{mult}}

\def\dim{\operatorname{dim}}

\def\NA{\operatorname{\overline{NA}}}

\def\NE{\operatorname{\overline{NE}}}
\def\Nef{\operatorname{Nef}}

\def\Specan{\operatorname{Specan}}

\theoremstyle{definition}
\newtheorem{definition}[theorem]{Definition}
\theoremstyle{definition}

\numberwithin{equation}{section}

\theoremstyle{remark}
\newtheorem{claim}[theorem]{Claim}

\author{Swapnajit Das}
\address{School of Maths, Tata Institute of Fundamental Research, Mumbai}
\email{das@math.tifr.res.in}

\author{Roktim Mascharak}
\address{Department of Mathematics, King's College London, Strand,
London WC2R 2LS, UK}
\email{roktim.mascharak@kcl.ac.uk}

\title[The log Sarkisov Program]{Log Sarkisov Program for Strongly $\mbQ$- factorial compact K\"ahler Threefolds}

\begin{document}
\maketitle
\tableofcontents
\section{Introduction}
The Minimal Model Program plays a central role in the birational classification of projective varieties. It predicts that for any smooth projective variety $X$, one should obtain either a good minimal model or a Mori fiber space after a sequence of birational contractions. Two different sequences of contractions may lead to two different outputs. Or in other words, the minimal model and Mori fiber space both are not necessarily unique. Therefore a natural problem is to understand how different outputs of the MMP are related. In the case of projective varieties, Kawamata proved that any two minimal models $(X,\Delta)$ and $(Y,\Gamma)$ are connected  by a sequence of $K_X+\Delta$-flops where $\sigma: X \dashrightarrow Y$ is the birational map and both pairs have klt singularities (see \cite{Kaw08}). Similarly, any two Mori fiber spaces are connected by standard Sarkisov links. This was proved in \cite{BM97} for dimension $3$ and in \cite{HM13} for arbitrary dimension.\\
The MMP was later extended to compact K\"ahler threefolds (see \cites{HP15, HP16,DH25}), and recently a breakthrough result  by Hacon-Xie in \cite{HX26}  was obtained for compact K\"ahler varieties of arbitrary dimension. A result similar to that in \cite{Kaw08} is also proved  in \cite[Theorem 3.27]{DHY26} for compact K\"ahler varieties of dimension $3$. Thus it is natural to expect that the Sarkisov Program should also hold for compact K\"ahler varieties. The purpose of this article is to establish a log Sarkisov Program for strongly $\mbQ$-factorial compact K\"ahler threefolds. More precisely, we prove the following:
 \begin{theorem}\label{SP}
     Let
      \item  \[
 \begin{tikzcd}
 (X,B_X) \arrow[r, dashrightarrow] \arrow{d}&  (X',B_{X'})\arrow{d}\\
S &  S' 
\end{tikzcd}
\] 
 be a bimeromorphic map between two log Mori fiber spaces of dimension three with strongly $\mbQ$-factorial klt singularities, induced by a log MMP relation between them. Then there exists a log bimeromorphic category $\mathcal{C}_{\theta}$ so that there is an algorithm to decompose $\Phi$ into a finite composite of four types of links in the log birational category $\mathcal{C}_{\theta}$, once all the values of $\theta$ stay $\leqslant \epsilon$ for some $0\leqslant \epsilon <1$.
 \end{theorem}
A Sarkisov link between two Mori fiber spaces $\phi:X \to S$ and $\psi: Y \to T$ is one of the following four types:

\begin{center}
\begin{tabular}{cccc}
\textbf{I} & \textbf{II} & \textbf{III} & \textbf{IV} \\
\begin{tikzcd}
Z \arrow[r, dashed] \arrow[d] & Y \arrow[d, "\psi"] \\
X \arrow[d, "\phi"] & T\arrow[dl] \\
S &
\end{tikzcd}
&
\begin{tikzcd}
Z \arrow[r, dashed] \arrow[d] & Z' \arrow[d] \\
X \arrow[d, "\phi"] & Y \arrow[d, "\psi"] \\
S \arrow[r, equal] & T
\end{tikzcd}
&
\begin{tikzcd}
    X \arrow[r, dashed] \arrow[d, "\phi"] & Z' \arrow[d] \\
    S \arrow[rd] & Y \arrow[d, "\psi"] \\
    & T
\end{tikzcd}
&
\begin{tikzcd}
X \arrow[rr, dashed] \arrow[d, "\phi"] & & Y \arrow[d, "\psi"] \\
S \arrow[rd, "s"] & & T \arrow["t", ld] \\
& R &
\end{tikzcd}
\\
\end{tabular}
\end{center}

\vspace{1 cm}
The proof follows the general strategy of the three-dimensional log Sarkisov program in \cite{BM97}, but several parts require new arguments in the K\"ahler setting. Since we have a K\"ahler class as a replacement of ample divisor, the homaloidal transform of a K\"ahler class is naturally a closed b-$(1,1)$-current. Thus it is necessary to work with generalized pairs whose moduli part is represented by a current. We then define Sarkisov degree. As the log Sarkisov program is already known for projective varieties, we explicitly deal with non-projective compact K\"ahler strongly $\mbQ$-factorial threefolds. We then decompose the bimeromorphic map $\Phi:X \dashrightarrow X'$ as a composition of standard Sarkisov links. Note that, we show that for strongly $\mathbb{Q}$-factorial non-projective K\"ahler three-folds links of type \textbf{IV} never appears in the Sarkisov factorization.\\
Our main challenge  is to show termination of any sequence of links. For this purpose, we prove a version of the ACC statement for glc thresholds on compact complex spaces  in the Fujiki class $\mathcal{C}$ with rational singularities. By using this ACC statement, we are able to prove the termination of links.\\

 \textbf{Organization of the article:} In Section \ref{sec:pre}, we recall some standard notations and definitions used in the K\"ahler minimal model program. In Section \ref{sec: decom}, we recall strongly $\mbQ$-factorial compact K\"ahler varieties, construct the log bimeromorphic category $\mathcal{C}_{\theta}$, define the Sarkisov degree and decompose the bimeromorphic map into Standard Sarkisov links. Section \ref{sec: ACC} establishes the required ACC statement for generalized log canonical thresholds with the moduli part being a current. Finally, in Section \ref{sec:term}, we use this ACC statement to prove the termination, thereby completing the proof of Theorem \ref{SP}.\\

\textbf{Acknowledgement:}
We would like to thank Prof. Omprokash Das and Prof. Calum Spicer for helpful discussions.The second author is supported by the EPSRC Centre for Doctoral Training in Geometry and Number Theory (LSGNT)\\

\textbf{AI Disclosure:}
The authors declare that no AI tools were used in the preparation of this manuscript.

 \section{Preliminaries}\label{sec:pre}
 An analytic variety $X$ is a reduced and irreducible complex space. A holomorphic map $f:X\rightarrow Y$ between two complex spaces is called a morphism. A small map between two analytic varieties is a bimeromorphic map $\phi:X\dashrightarrow X'$ between two normal analytic varieties such that there are closed analytic subsets $Z\subset X$ and $Z'\subset X'$ with $codim_X Z\geqslant 2$ and $codim_{X'}Z'\geqslant 2$ and $\phi|_{X\setminus Z}:X\setminus Z \rightarrow X'\setminus Z' $ is an isomorphism. A $(1,1)$ class $\alpha\in H^{1,1}_{BC}(X)$ is called general (resp. very general) if $\alpha$ is not contained in any finite union (resp. countable union) of analytic subvarieties of $H^{1,1}_{BC}(X)$.\\

 \begin{definition}
     Let $X$ be a normal analytic variety. The canonical sheaf is defined as $\omega_X:= (\wedge^{\dim X}\Omega^1_X)^{**}$. Unlike algebraic varieties, $\omega_X$ does not correspond to a Weil divisor $K_X$ such that $\omega_X=\mathcal{O}_X(K_X)$. But by an abuse of notation we will denote $\omega_X$ by $K_X$, and when we say $K_X$ is a canonical divisor we actually mean the canonical sheaf. This doesn't create a problem as we will be intersecting $\omega_X$ with subvarieties while running the MMP.\\
     A compact normal variety is called $\mbQ$-factorial if for every prime Weil divisor $D\subset X$ there is a $m\in \mbZ^+$ such that $mD$ is Cartier and there is a $k\in \mbZ^+$ such that $(\omega_X^{\otimes k})^{**}$ is a line bundle on $X$.\\
     For a normal analytic variety $X$ and an $\mbR$-divisor $B$ we say that $K_X+B$ is $\mbR$-Cartier, if locally around any point $x\in X$ we can choose a divisor $K_X$ such that $\mathcal{O}_X(K_X)=\omega_X$ and $K_X+B$ is $\mbR$-Cartier.  In this case we define the singularities of the pair $(X,B)$ as in \cite{KM}.
 \end{definition}
 \begin{definition}
     An analytic variety is called K\"ahler if there exists a positive closed real $(1,1)$ form $\omega\in \mathcal{A}^{1,1}_{\mbR}(X)$ such that the following holds: for every point $x\in X$ there exists an open neighbourhood $U$ of $x$ and a closed embedding $i_U:U\hookrightarrow V$ into an open set $V\subset \mbC^N$, and a strictly plurisubharmonic $\mathcal{C}^{\infty}$ function $f:V\rightarrow \mbR$ such that $\omega|_{U\cap X_{sm}}=(i\partial\overline{\partial}f)|_{U\cap X_{sm}}$. Here $X_{sm}$ is the smooth locus of $X$. 
 \end{definition}
 For the definitions of Bott-Chern cohomology class, $\NA(X)$ and the nef cone we refer to \cite[Theorem $2.2$]{DH25}.
 \begin{definition}
     If $X$ is a normal K\"ahler variety and $\omega\in H^{1,1}_{BC}(X)$, then we say that $\omega$ is a modified K\"ahler class if there exists a bimeromorphic morphism $\nu:X'\rightarrow X$ and a K\"ahler form $\omega'$ such that $\nu_*\omega'=\omega$.\\
     A closed positive bidegree $(1,1)$ current $T$ with local potential is called a K\"ahler current if $T\geqslant \omega$ for some $(1,1)$ K\"ahler form $\omega$. A class $\alpha\in H^{1,1}_{BC}(X)$ is called a big class if it contains a K\"ahler current.
 \end{definition}
 In this article we will face several instances where we will have to deal with relative bigness and nefness for $(1,1)$ classes. So next we define these.
 \begin{definition}
     Let $\pi:X\rightarrow S$ be a proper morphism of normal K\"ahler varieties such that $S$ is compact. Let $\beta$ be a closed $(1,1)$ current with local potentials. We say that the class $[\beta]$ is relatively K\"ahler (or K\"ahler over $S$) if $[\beta+\pi^*\omega_S]\in H^{1,1}_{BC}(X)$ is a K\"ahler class for some K\"ahler form $\omega_S$ on $S$, and we say that the class $[\beta]$ is relatively nef if $[\beta +\omega]$ is relatively K\"ahler for every relatively K\"ahler class $[\omega]$ on $X$.
 \end{definition}
 It is well known that if a class $\beta\in H^{1,1}_{BC}(X)$ is relatively K\"ahler (resp. relatively nef), then its restriction to each fiber is K\"ahler (resp. nef). By abuse of notation we will say that a closed bi-degree $(1,1)$ current $T$ with local potentials is relatively K\"ahler or relatively nef over $S$ if its class $[T]\in H^{1,1}_{BC}(X)$.\\
 To untwist the bimeromorphic maps between log Mori fiber spaces we will need to run MMP for generalized pair, which we will start defining in this set up.
 \begin{definition}
    Let $X$ be a normal analytic variety. A closed b-$(1,1)$ current $\beta$ is a collection of closed bidegree $(1,1)$ currents $\beta_{X'}$ on all proper bimeromorphic models $X'\rightarrow X$ such that if $p:X_1\rightarrow X_2$ is a bimeromorphic model, then $p_*\beta_{X_1}=\beta_{X_2}$.
 \end{definition}
Given a closed positive $(1,1)$ current $\beta$ on $X$ with local psh potentials, we can construct a well defined b-$(1,1)$ current as in \cite[Subsection $2.1$]{DHY26}. We move on to define generalized pairs.
\begin{definition}
    Let $f:X\rightarrow S$ be a proper morphism of normal K\"ahler varieties, with $S$ compact, $\nu:X'\rightarrow X$ a resolution, $B'$ a $\mbR$-divisor on $X'$ with simple normal crossings support such that $B:=\nu_*B'\geqslant 0$ and $\beta$ a closed b-$(1,1)$ current. We say that $(X,B+\beta)$ is a generalized pair if 
    \begin{enumerate}
        \item $\beta$ is a positive closed b-$(1,1)$ current that descends to $X'$,
        \item $[\beta_{X'}]\in H^{1,1}_{BC}(X')$ is nef over $S$, and
        \item $[K_{X'}+B'+\beta_{X'}]=\nu^*\gamma $ for some $\gamma\in H^{1,1}_{BC}(X).$
    \end{enumerate}
\end{definition}
In this article we will need to work with pairs with generalized log canonical singularities, which we define as in \cite[Definition $2.9$]{DHY26}.
\begin{definition}
    Let $P$ be a prime Weil divisor over $X$. We define the generalized discrepancy $a(P,X,B+\beta)$ by going to a log resolution $\nu:X'\rightarrow X$ of $(X,B+\beta)$ where $P\subset X'$ is a Weil divisor in $X'$ and $a(P,X,B+\beta):=-\mult_P(B')$.\\
    We say that a pair $(X,B+\beta)$ is generalized lc if for some log resolution $\nu':X'\rightarrow X$, we have $a(P,X,B+\beta)\geqslant -1$ for all prime divisors $P\subset X'$.
\end{definition}
For more details on generalized pairs and b-$(1,1)$ currents we refer to \cite{DHY26}.\\
\section{Sarkisov Degree and decomposition into Sarkisov links}\label{sec: decom}
In this section we develop the Sarkisov Program for K\"ahler threefold pairs with klt singularities. We will follow the treatment of \cite{BM97}. For the Sarkisov Program to work in a similar manner a priori it is not sufficient to assume $\mbQ$-factoriality of the involved varieties. We will need something stronger which we define below.
\begin{definition}
    A normal K\"ahler variety $X$ is said to be strongly $\mbQ$-factorial if for every reflexive sheaf $\mathcal{F}$ of rank one there exists an $m$ (depending upon $\mathcal{F}$ ) such that $(\mathcal{F}^{\otimes m})^{**}$ is a line bundle.
\end{definition}
The next lemma will help us to define the Sarkisov triplet in our setup, more precisely this will be the key to define the homaloidal transformation.
\begin{lemma}
    Let $f:X'\rightarrow X$ be a projective bimeromorphic morphism between strongly $\mbQ$-factorial K\"ahler varieties. If $[\phi']\in H^{1,1}_{BC}(X')$ then $[f_* \phi']\in H^{1,1}_{BC}(X)$.
\end{lemma}
\begin{proof}
Since $X$ is strongly $\mbQ$- factorial, $N^1(X'/X)= N^1(X')/f^*N^1(X)$ is generated by exceptional divisors. Hence $\phi'+E=0\in N^1(X'/X)$. But then $[\phi'+E]\in f^*N^1(X)$. So $[\phi'+E]=f^*\phi$ for some $\phi\in H^{1,1}_{BC}(X)$. And hence $[f_*\phi']=[f_*(\phi'+E)]=\phi$.
\end{proof}
Next we define what we mean by a MMP relation.
\begin{definition}
    Two or more K\"ahler varieties are log MMP related if and only if they are the end results of $K+B$-MMP over $\Specan\mbC$ starting from one log pair consisting of a nonsingular K\"ahler variety and a boundary divisor with only normal crossings.
\end{definition}
Next we define the category which should contain all our objects of study. This category is identical to its projective variety counterpart as in \cite{BM97}.
\begin{definition}\label{cat}
    We fix a  bimeromorphic class  and fix a variety $X$. Let $\theta$ be a function from the set of all prime divisors over $X$ to the real numbers $\in [0,1)$\\
    Note that the function $\theta$ is independent of the choice of $X$ in a fixed bimeromorphic class.
    A log bimeromorphic category associated to $\theta$ is the category of strongly $\mbQ$-factorial log K\"ahler varieties with klt singularities $(X,B_X=\sum d_k D_k)$ which satisfy the conditions 
    \begin{enumerate}
        \item $d_k=\theta(D_k)$
        \item $a(E,X,B_X)>-\theta(E)$ if $E$ is exceptional over $X$.
    \end{enumerate}
\end{definition}
Now our target is to replicate \cite[Lemma $3.6$]{BM97}. But let us first define the homaloidal transformation.
\begin{definition}\label{homaloidal}{\cite[Definition $3.8$]{BM97}}
    Let $\Phi$
      \item  \[
 \begin{tikzcd}
 (X,B_X) \arrow[r, dashrightarrow] \arrow{d}&  (X',B_{X'})\arrow{d}\\
S &  S' 
\end{tikzcd}
\] 
 be a map of Mori fiber spaces with strongly $\mbQ$-factorial compact K\"ahler threefolds, induced by a log MMP relation. Fix a K\"ahler form $\omega'$ on $S'$. We can choose $\mu'>0$ such that  $-\mu' (K_{X'}+B_{X'})+\phi^{'*} \omega' $ is also a K\"ahler form on $X'$. We define $\beta_{X'}:=-\mu' (K_{X'}+B_{X'})+\phi^{'*} \omega'$. Now we take a smooth compact K\"ahler variety $V$ which dominates both $X$ and $X'$, $X\xleftarrow{\sigma} V\xrightarrow{\sigma'} X'$, such that these maps are compatible with $\Phi$. We define the homaloidal transform $\beta_X$ to be $\sigma_* \sigma^{'*} \beta_{X'}$. Note that $\beta_X$ has local potentials.
    
\end{definition}

In the following lemma we justify our definition of log Sarkisov category (Definition \ref{cat}).
\begin{lemma}{\cite[Lemma $3.6$]{BM97}}
    Let
    \item  \[
 \begin{tikzcd}
 (X,B_X) \arrow[r, dashrightarrow] \arrow{d}&  (X',B_{X'})\arrow{d}\\
S &  S' 
\end{tikzcd}
\] 
 be a bimeromorphic map between two log Mori fiber spaces of strongly $\mbQ$-factorial compact K\"ahler $3$-folds having klt singularities, inducing the log MMP relation between them. Then 
 \begin{enumerate}
     \item there exists a log bimeromorphic category $\mathcal{C}_{\theta}$ containing both log Mori fiber spaces with the function $\theta$ satisfying the condition 
     \begin{enumerate}
         \item $\theta(D_k)=d_k(\leqslant \epsilon)$ and $\theta(D'_{k'})=d'_{k'}(\leqslant \epsilon)$.
         \item $\theta(E)=\epsilon$ if $E$ is exceptional both over $X$ and over $X'$.
     \end{enumerate}
     where $0\leqslant \epsilon<1$.
     \item $\mathcal{C}_{\theta}$ satisfies the following property: for any finite family $\lbrace (X_l,B_{X_l})\rbrace_{l\in L}\subset \mathcal{C}_{\theta}$ there exists an object $(Y,B_Y)\in \mathcal{C}_{\theta}$ such that $(Y,B_Y)$ dominates each $(X_l,B_{X_l})$ by a projective bimeromorphic morphism $(Y,B_Y)\dashrightarrow (X_l,B_{X_l})$ which is a process of $K+B$-MMP over $X_l$. 
 \end{enumerate}
\end{lemma}
\begin{proof}
    Before beginning the proof, note that as we have an MMP for a compact K\"ahler $\mbQ$-factorial klt pair $(X,B)$ due to \cite[Theorem $1.1$ and Theorem $1.2$]{DH25}, we automatically have \cite[Proposition $3.4$]{BM97} for free in this case. Now we start the proof.
    \begin{enumerate}
        \item  $(X,B_X)$ and $(X',B_{X'})$ are log MMP related, hence they satisfy \cite[Proposition $3.4$,(c)]{BM97}. Moreover there exists $0\leqslant \epsilon < 1 $ such that $a(D,X,B_X)\geqslant-\epsilon$ and $a(D,X',B_{X'})\geqslant-\epsilon$ for all prime divisors $D$, as $(X,B)$ and $(X',B_{X'})$ are both klt pairs. We set 
        \begin{enumerate}
            \item$\theta(D)=d_k$ or $d'_{k'}$ if $D=D_k $ or $D'_{k'}$.
            \item $\theta(D)=\epsilon$ if $D$ is exceptional over both $X$ and $X'$.
        \end{enumerate}
        When $D_k=D'_{k'}$ (i.e., one is strict transform of the other), then we have that $d_k=d'_{k'}$. Thus $\theta$ is well defined. When $E$ is not exceptional, over either $X$ or $X'$, and not coinciding with any of the $D_k$ and $D'_{k'}$, it is implicit that we set $\theta(E)=0$. It is clear from the construction that the log bimeromorphic category $\mathcal{C}_{\theta}$ contains both $(X,B_{X})$ and $(X',B_{X'})$.
        \item For a finite family $\lbrace (X_l,B_{X_l})\rbrace\subset \mathcal{C}_{\theta}$ we take a common resolution $Y$ which dominates each $X_l$ and both Mori fiber spaces $X$ and $X'$ by projective bimeromorphic morphisms so that the union of exceptional divisors and the strict transforms of $B_X$ and $B_{X'}$ is an SNC divisor. We set $B_{Y}=\sum \theta(D)D$, which by our choice of resolution is an SNC divisor. By the construction of $\theta$, $(Y,B_Y)\in \mathcal{C}_{\theta}$. Now consider the projective bimeromorphic morphism $f_l:(Y,B_Y)\rightarrow (X_l,B_{X_l})$. If we run a relative $(Y,B_Y)$-MMP over $(X_l,B_{X_l})$, it will terminate with a minimal model $(Y',B_{Y'})$ over $(X_l,B_{X_l})$. As $a_{li}>0$ in the equation \[
        K_Y+B_Y=f_l^*(K_{X_l}+B_{X_l})+\sum a_{li}E_{li}
        \], all $E_i$ get contracted by Negativity lemma (see\cite[Lemma 1.3]{Wang21}). Since $X_l$ and $Y$ both are strongly $\mbQ$-factorial and the induced map $Y'\to X_l$ is small, it follows that $Y'\cong X_l$. Hence $f_l$ is a process of $K+B$-MMP over $X_l$ (and thus over $\Specan \mbC$).
    \end{enumerate}
\end{proof}
\begin{remark}\label{rmk: Projective}
    Let $(X,B_X)\in \mathcal{C}_{\theta}$ be a projective pair. Consider any other element $(Y,B_Y)\in \mathcal{C}_{\theta}$. Since $(Y,B_Y)$ and $(X,B_X)$ are bimeromorphic, $Y$ is Moishezon. Since $Y$ is also K\"ahler, it follows that $(Y,B_Y)$ is also a projective pair by \cite[Theorem 1.6]{Nam02}. Therefore every element of $\mathcal{C}_{\theta}$ is projective if at least one element is projective. 
\end{remark}
Now we define the log Sarkisov degree in this case which is analogous to that for projective variety (as in \cite{BM97}).
\begin{definition}{\cite[Definition $3.8$]{BM97}}
    Consider the same setup as in the definition \ref{homaloidal}. The quasi-effective threshold $\mu$ is defined to be a positive real number such that $\mu (K_X+B)+ \beta_X\equiv 0$ over $S$.
    The value of $\mu=\frac{\beta_X\cdot R}{-(K_X+B_X)\cdot R}$\\
    The maximal multiplicity of an extremal ray (denoted by $\lambda$) is defined in the following way: $\frac{1}{\lambda}= sup \lbrace c\in \mbR_{>0}|(K_X+B_X)+c\beta_X$ is $\mathcal{C}_{\theta}$ generalized log canonical pair$\rbrace$.\\
    Or, in other words, for every divisor $F, a(F, X, B_X+c\beta_X)\geq -\theta(F)$.\\
    Moreover, if $\sup$ is $+\infty$, we define $\lambda=0$.
    Note that since $0\leq \epsilon<1$, the pair $(X,B_X+\frac{1}{\lambda}\beta_X)$ is a generalized klt pair.
    Accordingly, the number $e$ of the crepant exceptional divisors is defined to be \[
    e= \lbrace E; a(E,X,B_X+\frac{1}{\lambda} \beta_X )=-\theta (E) \text{ if } \lambda\neq0\rbrace\\
    \]
    Otherwise, $e$ is defined to be $0$.\\
    The triplet $(\mu, \lambda,e)$ is defined to be the log Sarkisov degree.
    
\end{definition}
Let $(Y, B_Y)\in \mathcal{C}_\theta$. Let $f: W\to Y$ be a log resolution such that $\beta_W$ is nef. 
\[ K_W= f^*(K_Y+B_Y)+ \sum a(E_i, Y, B_Y)E_i\]
\[ \beta_W+\sum m_{E}E = f^*\beta_Y\]
Observe that $m_E\geq 0$ by negativity lemma as $\beta_W$ is nef.
\[K_W+t\beta_W=f^*(K_Y+B_Y+t\beta_X)+\sum(a(E_i,Y, B_Y)E_i-tm_{E_i}E_i)\]
Thus we have 
\[ a(E_i, Y, B_Y+t\beta_Y)= a(E_i,Y,B_Y)-tm_{E_i}\]
Note that for any component of $B_Y$, the discrepancy and generalized discrepancy remain the same. The generalized discrepancy changes only for exceptional divisors. Moreover, $\frac{1}{\lambda_i}=min_{m_E>0} \frac{a(E,Y,B_Y)+\theta(E)}{m_E}$.
Furthermore, we define the local generalized $\mathcal{C}_{\theta}$-threshold at $x$, $c_\theta(x,Y)=min_{m_E>0, x\in f(E)} \frac{a(E,Y,B_Y)+\theta(E)}{m_E}$.
\begin{remark}\label{rmk:local threshold theta}
Observe that for any fixed generalized pair $(Y, B_Y+\beta_Y)$ where $(Y, B_Y)\in \mathcal{C}_\theta$, local generalized $c_\theta$ threshold can take only finitely many values.
\end{remark}

So as we can see the log Sarkisov degree is highly dependent on the category $\mathcal{C}_{\theta}$.\\
We closely follow the techniques from \cite[Theorem $3.7$]{BM97}. The proof has two parts. First, we decompose the map $\Phi$ into Sarkisov links. The second step is the termination of the links. The proof of both of these steps is dependent on the study of the cases:
\begin{enumerate}
    \item $\lambda\leqslant \mu$ 
    \item $\lambda > \mu$
\end{enumerate}
 Obviously there are several technical challenges which we need to overcome here. We slightly modify our approach and explain it here. Since the log Sarkisov program is already known for projective threefolds, by Remark \ref{rmk: Projective}, we may assume that both $(X,B_X)$ and $(X',B_X')$ are non-projective compact K\"ahler klt pairs of dimension $3$. Note that if $\dim S\leq 1$, then $S$ is projective. This implies that $X$ is projective. Thus we may assume that $\dim S=\dim S'=2$.\\
Let $\Phi: X \dashrightarrow X'$ be the bimeromorphic map. We will decompose $\Phi$ into standard Sarkisov links.\\

Before we get into the proof, we will need a very important lemma to study the first case.

\begin{lemma}[Noether-Fano Criterion]\label{NF}{\cite[Theorem $4.2$]{CO95}}
The bimeromorphic map $\Phi$ in \ref{SP} is an isomorphism of Mori fiber spaces if $\lambda\leqslant \mu$ and $K_X+B_X+\frac{1}{\mu} \beta_X$ is nef.
\end{lemma}
\begin{proof}
    A similar proof to that of \cite[Theorem $4.2$]{CO95} works here since a K\"ahler form shows same intersection theoretic behaviour as an ample divisor, replace $K$ by $K+B$ and the ample divisor by K\"ahler form. For the reader's convenience, we include the proof here.
 Let $W$ be a common resolution that resolves the bimeromorphic map $\Phi:X \dashrightarrow X'$. Since $\dim S=\dim S'=2$, the general fibers are isomorphic and hence $\mu_X=\mu_{X'}=\mu$. Write 
  \[ K_W+\frac{1}{\mu}\beta_Z= q^*(K_{X'}+B_{X'}+\frac{1}{\mu}\beta_{X'})+\sum a(E_i,X',B_{X'}+\frac{1}{\mu}\beta_{X'})E_i\]
Since $\beta_{X'}$ descends on $X'$, $a(E_i,X',B_{X'})=a(E_i,X',B_{X'}+\frac{1}{\mu}\beta_{X'})$ for each $i$.
\[K_W+\sum \theta(F)F +\frac{1}{\mu}\beta_Z=q^*(K_{X'}+B_{X'}+\frac{1}{\mu}\beta_{X'})+E_q\]
where $F\subset W$ runs over all prime divisors. Since $\theta$ takes nonzero value only on finitely many divisors, this is a finite sum.
By definition of $\theta$, we know that $a(D,X',B_{X'})=-\theta(D)$ for any boundary prime divisor and $a(E,X',B_{X'})> -\theta(E)$. Thus it follows that $E_q$ is effective and its support is the union of all $q$-exceptional divisors.
Write 
\[K_W+\sum \theta(F)F +\frac{1}{\mu}\beta_Z=p^*(K_{X}+B_{X}+\frac{1}{\mu}\beta_{X})+E_p\]
Since $(X,B_X+\frac{1}{\mu}\beta_X)\in \mathcal{C}_{\theta}$, $E_p\geq 0$.
Now $H=E_p- E_q=q^*(K_{X'}+B_{X'}+\frac{1}{\mu}\beta_{X'})- p^*(K_{X}+B_{X}+\frac{1}{\mu}\beta_{X})$.
By the Negativity lemma \cite[Lemma 1.3]{Wang21}, $H=0$. Hence $p^*(K_X+B_X+\frac{1}{\mu}\beta_X)=q^*(K_{X'}+B_{X'}+\frac{1}{\mu}\beta_{X'})$. Let $(\phi'\circ q): W\to S'$ be the composition. Clearly  $q^*(K_{X'}+B_{X'}+\frac{1}{\mu}\beta_{X'})=(\phi'\circ q)^*\omega'=(\phi\circ p)^*(K_S+B_S+\frac{1}{\mu}\beta_S)$. Both $\phi\circ p$ and $\phi'\circ q$ are projective morphisms. Let $C$ be any curve contracted by $\phi\circ p$. Then $\omega'\cdot (\phi'\circ q)_*C =0$ and hence, $\phi'\circ q$ contracts $C$. Thus by the rigidity lemma we have a morphism $h: S \to S'$. 
Note that $\rho(S') \leq \rho(S)$ and hence $\rho(X')\leq \rho(X)$.
Moreover, we have $\rho(W)=\rho(X)+e_p=\rho(X')+e_q$ where $e_p$ and $e_q$ denote the numbers of $p$ and $q$ exceptional divisors respectively. Let $M$ be any $q$-exceptional divisor.  $E_p=E_q$ implies that $M\in E_p$.  The coefficient of $M$ in $E_p=a(M, X,B_X+\frac{1}{\mu}\beta_X)+\theta(M)>0$. If $M$ is not exceptional then it must be strict transform of some boundary component $N$ on $X$. But then $a(M, X, B_X)>-\theta(M)$ which is a contradiction. Thus $M$ must be $p$-exceptional and hence, $\rho(X')\geq \rho(X)$. Therefore $\rho(X)=\rho(X')$ and $\rho(S)=\rho(S')$.
Thus it follows that $h$ is an isomorphism. By \cite[Lemma 2.16]{DHY26}, $\Phi$ is also an isomorphism.

\end{proof}

\textbf{Decomposition of $\Phi$ into standard Sarkisov links:}\\
Let $(\mu,\lambda,\nu)$ be the Sarkisov triplet corresponding to $\phi:X\rightarrow S$.
\begin{enumerate}
    \item Suppose we are in the case $\lambda\leqslant \mu$. 
First we ask whether $K_X+B_X+\frac{1}{\mu} \beta_X$ is nef. If yes, then by Lemma \ref{NF}, we have that $\Phi$ is an isomorphism of Mori fiber spaces.\\
If $K_X+B_X+\frac{1}{\mu} \beta_{X}$ is not nef, we construct a normal compact K\"ahler variety $T$ dominated by $S\rightarrow T$ such that $K_X+B_X+\frac{1}{\mu} \beta_X$ is not relatively nef over $T$. Let $R=\mbR^{\geq0}[C]$ be the $K_X+B_X$-negative extrmal ray giving the Mori fiber space $\phi:X\rightarrow S$. By the cone theorem we know that the extremal rays do not accumulate around $R$. As $(K_X+B_X+\frac{1}{\mu} \beta_X)\cdot R=0$, we can choose $F=P+R$, a 2 dimensional extrmeal face, where $P$ is a $K_X+B_X+\frac{1}{\mu}\beta_X$-negative extremal ray. $F$ is $K_X+B_X+(\frac{1}{\mu}-\epsilon)\beta_X$-negative for sufficiently small $\epsilon$.\\
\begin{claim}\label{clm:cont}
    There exists a compact K\"ahler variety $T$ together with a projective morphism with connected fibers $cont_F:X \to T$ which contracts the extremal face $F$ defined above.
\end{claim}
\begin{proof}[Proof of \ref{clm:cont}]
    Since $\phi:X \to S$ is a contraction of a negative extremal ray and $K_X+B_X+\frac{1}{\mu}\beta_X\equiv_S 0$, $K_X+B_X+\frac{1}{\mu}\beta_X=\phi^*\alpha$ for some $\alpha \in H^{1,1}_{BC}(S)$. By the cone theorem \cite[Theorem 1.6]{DHY26}, we know that $P=\mbR^+[\Gamma]$. By the canonical bundle formula \cite[Theorem 0.3]{HP24}, $K_X+B_X+\frac{1}{\mu}\beta_X =\phi^*(K_S+B_S+\beta_S)$  where $(S,B_S+\beta_S)$ is a generalized klt pair. Now $(K_X+B_X+\frac{1}{\mu}\beta_X)\cdot \Gamma =(K_S+B_S+\beta_S)\cdot \phi_*\Gamma< 0$.
Note that $\phi_*P=\mbR^{\geq0}[\phi_*\Gamma]$. 
\begin{claim} \label{clm:extremal}
    $\phi_*P$ is also an extremal ray in $\NA(S)$.
\end{claim}
\begin{proof}[Proof of claim \ref{clm:extremal}]
    We know that $\phi_*\NA(X)=\NA(S)$ and $\Nef(S)=(\NA(S))^*$. Moreover, we also have $\ker(\phi_* :N_1(X) \to N_1(S))=R$. 
Let $a,b \in \NA(S)$ such that $a+b \in \phi_*P$. Write $a+b=\phi_*c$ where $c\in P$. Then
$\phi_*a'+\phi_*b'=\phi_*c \implies \phi_*(a'+b'-c)=0\implies a'+b'-c =\lambda [C]$ where $\lambda\geq 0$.\\
Now $a'+b'\in c+\lambda [C]\implies a'+b'\in F$. Since $F$ is extremal, $a',b'\in F\implies a,b\in \phi_*P$. Hence our claim follows.
\end{proof}
    Thus by the contraction theorem we can contract $\phi_*\Gamma$ and obtain a projective morphism with connected fibers $g:S \to T$. Thus it follows that $g\circ \phi:X \to T$ is a projective morphism which contracts $F$. Moreover, $\dim T=2$ (otherwise, it would imply that $X$ is projective) which implies that $g:S \to T$ is bimeromorphic.
\end{proof}
Now, since $g:S \to T$ is bimeromorphic, the general fibers of $g\circ \phi$ and $\phi$ are isomorphic. Again, from the construction, we know that $(K_X+B_X+\frac{1}{\mu}\beta_X)\cdot X_s=0$ for any general fiber $X_s$ of $\phi$. Thus $K_X+B_X+\frac{1}{\mu}\beta_X$ is pseudo- effective over $T$. 
Note that $\beta_X$ is big on $X$. By \cite[Theorem 1.1]{HX26}, we may run $K_X+B_X+\frac{1}{\mu}\beta_X$-MMP over $T$ and this terminates with a minimal model over $T$.\\

Then we ask whether the first non-flipping contraction is divisorial. If the answer is yes, then the $K+B+\frac{1}{\mu}\beta$-MMP consists of a sequence of $K+B+\frac{1}{\mu}\beta$-flips $X\dashrightarrow Z'$ followed by a $K+B+\frac{1}{\mu}\beta$-negative divisorial contraction $Z'\rightarrow X_1$. $\rho(X_1/T)=1$. Since the image of exceptional locus has codimension at least $2$, it does not dominate the base $S_1$. Thus, for a general fiber $X_{1t}$, $(K_{X_1}+B_{X_1}+\frac{1}{\mu}\beta_{X_1})\cdot X_{1t}=0$. Since $\beta_{X_1}$ is big, $\beta_{X_1}\cdot X_{1t} >0$ and hence $(K_{X_1}+B_{X_1})\cdot X_{1t}<0$. Thus it follows that $-(K_{X_1}+B_{X_1})$ is relatively ample over $S_1=T$. So we have a link of type III.
\[
\begin{tikzcd}
X \arrow[r, dashrightarrow] \arrow{d} & Z' \arrow{rd} \\
S \arrow{rd} & & X_1 \arrow{d}\\
& T \arrow[r,"\cong"] & S_1
\end{tikzcd}
\]
\begin{claim}\label{clm:div-con}
    The above $K_X+B_X+\frac{1}{\mu}\beta_X$-MMP over $T$ always contracts a divisor.
\end{claim}
\begin{proof}[Proof of claim \ref{clm:div-con}]
    Let $C$ be the curve contracted by the bimeromorphic morphism $g:S \to T$. $(K_S+B_S+\beta_S)=g^*(K_T+B_T+\beta_T)+ aC$ where $a=\frac{(K_S+B_S+\beta_S)\cdot C}{C^2}>0$. Thus we have 
    \[K_X+B_X+\frac{1}{\mu}\beta_X=\phi^*g^*(K_T+B_T+\beta_T)+a\phi^*C\]
    \[ K_X+B_X+\frac{1}{\mu}\beta_X\equiv_T G\] where $G=a\phi^*C>0$.
    Let $G_{X_1}$ be the strict transform of $G$ on the output $X_1$. Then $G_{X_1}$ is effective, exceptional over $T$ and nef over $T$. It follows from \cite[Lemma 2.8]{DH24} that $G_{X_1}=0$. In particular, the divisor $G$ is contracted by the sequence of MMP steps.
\end{proof}
Note that the above argument uses the fact that the base of the Mori fiber space has dimension $2$ and that's why a link of type IV does not occur in this case.
  One more very important observation is that in this case the quasi-effective threshold remains the same. Since MMP preserves the singularities, $(X,B_X+\frac{1}{\mu}\beta_X)$ is $\mathcal{C}_{\theta}$-generalized log canonical and so is $(X_1, B_{X_1}+\frac{1}{\mu}\beta_{X_1})$. Thus $\lambda_1\leq \mu=\mu_1$. Therefore untwisting by link of type III, we stay in the case $\lambda\leq \mu$.\\

\[
\]
    \item Now we move on to the case of $\lambda>\mu$. In this case we construct a maximal divisorial blowup $p:Z\rightarrow X$ with respect to $K_X+B_X+\frac{1}{\lambda}\beta_X$, that is, $\rho(Z/X)=1$, the exceptional locus of $p$ is a prime divisor $E$, and  $p$ is $K_X+B_X+\frac{1}{\lambda} \beta_X$-crepant. \\
    
Let $I$ be the finite set containing all crepant exceptional divisors. Recall that $a(E, X,B_X+\frac{1}{\lambda}\beta_X)=-\theta(E)\leq 0$. Then by \cite[Proposition 2.28]{HX26}, we can construct $\eta: \bar{X} \to X$ such that $\bar{X}$ is strongly $\mbQ$-factorial, $\eta$ is projective and the exceptional divisors are exactly the divisors contained in $I$. The following equation is satisfied:

\[K_{\bar{X}}+B_{\bar{X}}+\frac{1}{\lambda}\beta_{\bar{X}}=\eta^*(K_X+B_X+\frac{1}{\lambda}\beta_X).\]
Where $B_{\bar{X}}=\eta_*^{-1}B_{X}+\sum\theta(E)E$. Ignoring the moduli part $\beta$ we write:
    \[ K_{\bar{X}}+\eta_*^{-1}B_X+ \sum_{E\in I}\theta(E)E=\eta^*(K_X+B_X)+\sum_{E \in I}(a(E,X,B_X)+\theta(E))E\]
    \[K_{\bar{X}}+B_{\bar{X}}=\eta^*(K_X+B_X)+ \sum_{E \in I}(a(E,X,B_X)+\theta(E))E\]

Since $K_{\bar{X}}+B_{\bar{X}}+\frac{1}{\lambda}\beta_{\bar{X}}$ is gklt, $(\bar{X}, B_{\bar{X}})$ is a klt pair. We may run $K_{\bar{X}}+B_{\bar{X}}$-MMP over $X$ and obtain the following commutative diagram:

\[\begin{tikzcd}
	{\bar{X}} && Y \\
	& X
	\arrow[dashed, from=1-1, to=1-3]
	\arrow["\eta"', from=1-1, to=2-2]
	\arrow["{\eta'}", from=1-3, to=2-2]
\end{tikzcd}\]

such that $\eta': Y \to X$ is small. Any small projective bimeromorphic morphism with connected fibers between two strongly $\mbQ$-factorial varieties is an isomorphism. In particular, $\eta'$ is an isomorphism. Thus the last divisorial contraction in the above MMP is the required maximal divisorial blowup. This satisfies the following:
$p:Z \to X$ such that $\rho(Z/X)=1$ and the exceptional locus is precisely $E$ and $K_Z+B_Z+\frac{1}{\lambda}\beta_Z=p^*(K_X+B_X+\frac{1}{\lambda}\beta_X)$.

\begin{claim}
    $K_Z+B_Z+\frac{1}{\lambda}\beta_Z$ is not pseudo-effective over $S$.
\end{claim}
\begin{proof}
    $K_Z+B_Z+\frac{1}{\lambda}\beta_Z=p^*(K_X+B_X+\frac{1}{\lambda}\beta_X)$. It is enough to show $K_X+B_X+\frac{1}{\lambda}\beta_X$ is not pseudo-effective over $S$. This is obvious as $\lambda>\mu$.
\end{proof}    
 Therefore, we may run $K_Z+B_Z+\frac{1}{\lambda}\beta_Z$-MMP over $S$ which terminates with a Mori fiber space. 
 

 \[
 \] 

  If the first non-flipping contraction in this MMP is divisorial, then the $K+B+\frac{1}{\lambda}\beta$-MMP with the scaling of a very general K\"ahler class $\omega$, consists of a sequence of $K+B+\frac{1}{\lambda}\beta$-flips $X\dashrightarrow Z'$ followed by a $K+B+\frac{1}{\lambda}\beta$-negative contraction $Z'\rightarrow X_1$. Since $\rho(X_1/S)=1$ and $X_1\rightarrow S_1=S$ is a $K_{X_1}+B_{X_1}+\frac{1}{\lambda}\beta_{X_1}$-negative contraction and thus a $K_{X_1}+B_{X_1}$-negative fiber space as $\beta_{X_1}$ is big on $X_1$. So we have a link of type II.
 \[
 \begin{tikzcd}
& Z \arrow[r, dashrightarrow] \arrow{ld}&  Z'\arrow{rd}&\\
X \arrow{d}& & & X_1 \arrow{d}\\
S & & & S_1 \arrow{lll}
\end{tikzcd}
\] 

\begin{claim}\label{clm:E survive}
    The contraction $Z'\to X_1$ cannot contract the transform of $E$ on $Z'$ in the link of type II.
\end{claim}
\begin{proof}[Proof of claim \ref{clm:E survive}]
    If possible, let us assume that $E$ is contracted. Then the induced bimeromorphic map $X\dashrightarrow X_1$ would be small. Note that both $-(K_X+B_X)$ and $-(K_{X_1}+B_{X_1})$ are ample over $S$. Thus it follows that $X\dashrightarrow X_1$ is an isomorphism by \cite[Lemma 2.16]{DHY26}. Note that $a(E,X,B_X+\frac{1}{\lambda}\beta_X)= -\theta(E)$. Since $Z'\to X_1$ is $(K_{Z'}+B_{Z'}+\frac{1}{\lambda}\beta_{Z'})$-negative, $a(E,X_1,B_{X_1}+\frac{1}{\lambda}\beta_{X_1})> -\theta(E)$. This is a contradiction. Hence $E$ survives as a divisor on $X_1$.
\end{proof}

Now, if the first non-flipping contraction is not a divisorial contraction, then $K+B+\frac{1}{\lambda}\beta$-MMP consists of a sequence of $K+B+\frac{1}{\lambda}\beta$-flips $X\dashrightarrow Z'$ followed by a $K_{X_1}+B_{X_1}$-negative Mori fiber space contraction $Z'=X_1\rightarrow S_1$. Observe that it is in fact a $K_{X_1}+B_{X_1}$-negative Mori fiber space contraction. Since $\rho(X_1/S)=\rho(Z/S)=2$, we have $\rho(S_1/S)=1$, so we have a link of type I.
 \[
 \begin{tikzcd}
& Z \arrow[r, dashrightarrow] \arrow{ld}&  X_1\arrow{dd}\\
X \arrow{d}& &  \\
S & &  S_1\arrow{ll}
\end{tikzcd}
\]

Doing the exact same analysis as \cite{BM97} we get that the each component of the Sarkisov triplet is non-increasing. In fact  $\mu$ remains stationary and both of $S$ and $S_1$ are surfaces, hence $X_1$ and $X$ are isomorphic over a large (Zariski) open set of $S$. Similarly, as in the projective threefold case, we can see (by virtue of MMP and the Negativity lemma (\cite[Lemma 1.3]{Wang21}) that $\lambda_1\leqslant \lambda$ and when $\lambda_1=\lambda$ then $e_1<e$.\\
So we go back to the start with a strictly decreased Sarkisov degree in the case of untwisting by  a link of type I or type II.\\
\end{enumerate}
We have shown that $\Phi$ can be untwisted in terms of Sarkisov links. Now we need to show the termination of the sequence of links.\\

\section{A version of ACC for generalized log canonical threshold in complex analytic setting}\label{sec: ACC}
In this section, we prove a version of ACC for generalized thresholds in the complex analytic setting with the moduli part being a current. This will be used in our proof of the termination of links. Before proving the theorem, we recall a definition.
\begin{definition}\cite[Definition 2.43]{HX26}\label{def:NQC}
    Let $X$ be a normal compact complex analytic space in the Fujiki class $\mathcal{C}$ with rational singularities. Let $\alpha\in H_{BC}^{1,1}(X)$ be a nef class. Then we say $\alpha$ is NQC if there exist $a_i\in \mbR^{>0}$ and $\alpha_i\in H^2(X,\mbQ)$ such that $\alpha=\sum a_i\alpha_i$ under the canonical inclusion $H_{BC}^{1,1}(X)\to H^2(X,\mbR)$ such that $\alpha_i$ is non-negative on $\NA(X)$.
\end{definition}
The following theorem is a version of ACC statement which will be useful to prove the termination of links.
More precisely,
\begin{theorem}\label{thm:ACC}
    Let $\Lambda$ be a DCC set of non-negative real numbers and $n$ be a fixed positive integer. We consider a collection of complex analytic spaces satisfying the following:
    \begin{enumerate}
        \item $(X,B+\gamma)$ is a generalized log canonical pair of dimension $n$ such that coefficients of $B$ belong to $\Lambda$ and $\beta$ is a closed b-(1,1)-current on $X$ where $X$ is a normal compact complex analytic space in the Fujiki class $\mathcal{C}$ with rational singularities.
        \item  Suppose $f:Y\to X$ is a log resolution of $(X,B+\gamma)$ such that both $\beta_Y$ and $\gamma_Y$ are NQC classes (see Definition \ref{def:NQC}) and descend on $Y$.  Assume that $\beta_Y=\sum a_i (\beta_Y)_i, \gamma_Y=\sum b_j (\gamma_Y)_j$  where $a_i, b_j\in \Lambda$  and $ (\beta_Y)_i, (\gamma_Y)_j\in H^2(Y,\mbZ)$  for all $i,j$ and each $(\beta_Y)_i$ and $(\gamma_Y)_i$ is non-negative on $\NA(Y)$.
        \item  Let $D\geq 0$ be an $\mathbb{R}$-Cartier $\mbR$-divisor and the coefficients of $D$ are contained in $\Lambda$.
    \end{enumerate}
 Then $\tau=\sup\{ t\geq 0 | K_X+B+\gamma+t(D+\beta)$ is a generalized lc pair $\}$ belongs to an ACC set $\Theta$.
\end{theorem}
\begin{proof}
Cover $X$ by finitely many Stein open subsets, i.e., $X=U_1 \cup U_2\cup \cdots \cup U_d$. Fix any $U_i=U$ and define $f_U: V=f^{-1}(U) \to U$ to be the induced morphism. Since $X$ has rational singularities, $U$ has also rational singularities. Hence it follows that $H^i(V,\mathcal{O}_V)=H^i(U,\mathcal{O}_U)=0$ for all $i>0$ as $U$ is Stein. 
Let $\tau_U$ be the glc threshold for the pair $(U, B_U+\gamma_U)$ with respect to $(D_U+\beta_U)$ where $K_U+B_U+\gamma_U=(K_X+B_X+\gamma)|_U$ and $(D_U+\beta_U)=(D+\beta)|_U$.

Write $\gamma_V= \sum b_j (\gamma_V)_j$ where $b_j\in \Lambda$ and $(\gamma_V)_j\in H^2(V,\mbZ)$. From  the long exact sequence in cohomology associated to the exponential exact sequence, we obtain that $(\gamma_V)_i= c_1(L_i)$ where $L_i$ is a holomorphic line bundle on $V$ as $H^2(V,\mathcal{O}_V)=0$. Since $U$ is Stein and $f|_U$ is projective, we can say $L_i$ is represented by a divisor $M_i'$. In particular, we replace the class $\gamma_V=\sum a_i M_i'=M'$ where $M'$ is a divisor on $V$. Similarly, $\beta_V$ is also represented by some divisor $P'$ on $V$. Note that $(f_U)_*(\beta_V)=(\beta_U)$. Let $M$ and $P$ be two divisors on $U$ defined by $M:=(f_U)_*M', P:=(f_U)_*P'$. We may shrink $U$, if necessary, and assume $U$ is of finite CW type. Let
\[\tau_U=\sup\{t\geq 0 | K_U+B_U+\gamma_U+t(D_U+\beta_U)\text{is generalized log canonical} \}.\] 

Consider the following commutative diagram:
\[\begin{tikzcd}
	V && Y \\
	U && X
	\arrow["{i_V}", hook, from=1-1, to=1-3]
	\arrow["{f_U}", from=1-1, to=2-1]
	\arrow["f", from=1-3, to=2-3]
	\arrow["{i_U}"', hook, from=2-1, to=2-3]
\end{tikzcd}\]
Now note that $H^1(U,\mathcal{O}_U)=H^2(U,\mathcal{O}_U)=0$. Thus it follows that $\Pic (U)\cong H^2(U,\mbZ)$, and hence $\Pic(U)\otimes \mbR\cong H^2(U,\mbZ)\otimes \mbR\cong H^2(U,\mbR)$.
From the above diagram, we obtain a local reduction such that $f_U:V \to U$ is the log resolution of $(U, B_U, M)$ satisfying $K_V+B_V+M'=(f_U)^*(K_U+B_U+M)$ , $(K_U+B_U+M)$ is $\mbR$-Cartier and $M,P$ are divisors on $U$ whose Chern classes are given by $\gamma_U$ and $\beta_U$ respectively.
\begin{claim}\label{clm: glc-currrnt-divisor same}
    $\tau_U=\sup\{t\geq 0 | K_U+B_U+M+t(D_U+P)$ is generalized log canonical$\}$.
\end{claim}
\begin{proof}[Proof of Claim \ref{clm: glc-currrnt-divisor same}]
    Fix a $t_0$. Choose an $\mbR$- cartier $\mbR$-divisor $A_U$ on $U$ whose Chern class is given by $(K_U+B_U+\gamma_U+t_0(D_U+\beta_U))$. Let $A_V:=(K_V+B_V+M'+t_0(D_V+P'))$ whose Chern class is same as $(K_V+B_V+\gamma_V+t_0(D_V+\beta_V))$. Note that $c_1(A_V-(f_U)^*A_U)=0$.
    Note that $E= A_V-(f_U)^*(K_U+B_U+M+t_0(D_U+P))$ is exceptional divisor and $\mbR$- linearly trivial over $U$. By negativity lemma, we have $E=0$.
    Hence $A_V=(f_U)^*(K_U+B_U+M+t_0(D_U+P))$. Thus both the discrepancy for the current generalized pair $(U,B_U+t_0D_U+\gamma_U+t_0\beta_U)$ and for the divisor generalized pair $(U, B_U+t_0D_U+M+t_0P)$ are the same. Hence our claim follows.
\end{proof}

Therefore the local pair on $U$ is a generalized log canonical pair and it satisfies the hypothesis of \cite[Theorem 3.3]{HX26B} taking the DCC set to be $\Lambda \cup \{0\}$. Therefore it follows that $\tau_U \in \Theta$ where $\Theta$ is an ACC set. Observe that $\tau=\min \{\tau_U | U \in \{U_1,U_2,...,U_d\}\}$. Let $\Theta'$ be a subset of $\Theta$ containing all such $\tau$. Then $\Theta'$ must satisfy ACC. In particular, ACC for generalized lc thresholds holds in this case.
\end{proof}

\section{Termination of Sarkisov Links}\label{sec:term}
In this section, we show any such sequence of Sarkisov links that decomposes the bimeromorphic map $\Phi:X\dashrightarrow X'$, will terminate. We begin with an important lemma.\\

\begin{lemma}\label{lem:co-betaX'}
     There are positive real numbers $b_1,b_2,...,b_s$ and classes $\eta_j\in H^2(X',\mbZ)$ such that $\beta_{X'}=\sum_{j=1}^s b_j\eta_j$ and each $\eta_j$ is non-negative on curve classes of $X'$.
 \end{lemma}
\begin{proof}
    Choose $T>>0$ such that $\omega=T\beta_{X'}-(K_{X'}+B_{X'})$ is K\"ahler. By \cite[Lemma 2.45]{HX26} $\beta_{X'}$ is an NQC class in the sense of Definition \ref{def:NQC}. We may write 
    \[ \beta_{X'}=\sum_{i=1}^s r_j\eta'_j, r_j>0, \eta'_j \in H^2(X',\mathbb{Q})\]
where each $\eta_j$ is nonnegative on $\NA(X')$, and thus on curve classes as well. Choose $n_j>0$ such that $\eta_j= n_j\eta'_j\in H^2(X',\mbZ)$. Thus we obtain our desired result where $b_j=\frac{r_j}{n_j}$ for all $j\in\{1,2,...,s\}$.
\end{proof}

\textbf{Claim $1$}:\textit{There cannot be infinitely many untwistings by links in the case of $\lambda \leqslant \mu$. }\\
\begin{proof}
Note that we are only interested in non-projective compact K\"ahler threefolds. In this case, we obtain a sequence of links of  type III. Since a link of type III contracts a divisor, the Picard number drops by $1$. Hence we can only have a finite number of such untwistings.\\

\end{proof}
\textbf{Claim $2$}: \textit{There is no infinite sequence of successive untwistings by the links in the case $\lambda>\mu$}. 
\begin{proof}
Observe that the quasi-effective threshold remains the same throughout the entire process. Let $\mu_0$ be the fixed pseudo-effective threshold. Note that if $\lambda_i=\lambda_{i+1}$ then $e_{i+1}<e_{i}$ by the negativity lemma (see \cite[Lemma 1.3]{Wang21}), and hence we will be done. So we may assume that $\lambda_i$ is not stationary. So we have a non-decreasing sequence $\lbrace \frac{1}{\lambda_i}\rbrace$ which is bounded by $\frac{1}{\mu_0}$. Let $\alpha$ be its limit. We may assume that $\frac{1}{\lambda_i}< \alpha$ for sufficiently large $i$. 
\begin{claim}\label{clm:glc-i}
    $(X_i,B_{X_i}+\alpha\beta_{X_i})$ is generalized log canonical for every sufficiently large $i$.
\end{claim}
\begin{proof}[Proof of claim \ref{clm:glc-i}]
    Let $\tau_i$ be the glc threshold of the pair $(X_i,B_{X_i})$ with respect to $\beta_{X_i}$. Note that $\tau_i>\frac{1}{\lambda_i}$. For each $i$, take a common resolution $W_i$ which resolves $X_i\dashrightarrow X'$. Recall that $\beta_{X'}$ is K\"ahler on $X'$. In particular, $\beta_{X'}$ is an NQC class. Then by \cite[Lemma 2.44]{HX26}, we can say $q^*\beta_{X'}$ is also an NQC class.  Use the notations $B_{X_i}= B, \gamma=0, D=0, \beta_{X_i}=\beta$ and $\Lambda= \{0,\epsilon, b_1,....,b_s\} \cup \{$ Coefficients of $B_X\}\cup \{$ Coefficients of $B_{X'}\}$ where $b_i$ are the coefficients appearing in Lemma \ref{lem:co-betaX'}. As $\Lambda$ is finite, it satisfies DCC. By Theorem \ref{thm:ACC}, we obtain that $\tau_i$ belongs to an ACC set. If there exist infinitely many $\tau_i$ strictly smaller than $\alpha$, choose one $\tau_{i_1}$. Since $\frac{1}{\lambda_i}\to \alpha$, we can find $i_2$ such that $\tau_{i_1}<\frac{1}{\lambda_{i_2}}<\tau_{i_2}$. Repeating this process, we obtain a strictly increasing subsequence of $\{\tau_i\}$, which is a contradiction. Therefore $\tau_i\geq \alpha$ for sufficiently large $i$. This proves our claim. 
\end{proof}
\begin{claim}\label{clm: improve-dis}
    The maximal divisorial blowup and all bimeromorphic steps of the links improve $\alpha$-discrepancies, i.e discrepacy with respect to the generalized boundaries $B+\alpha\beta$.
\end{claim}
\begin{proof}[Proof of claim \ref{clm: improve-dis}]
For a maximal divisorial blowup $p_i:Z_i \to X_i$, 
 \[ K_{Z_i}+B_{Z_i}+\alpha\beta_{Z_i}= K_{Z_i}+B_{Z_i}+\frac{1}{\lambda_i}\beta_{Z_i}+(\alpha-\frac{1}{\lambda_{i}})\beta_{Z_i}\]
 \[=(p_i)^*(K_{X_i}+B_{X_i}+\frac{1}{\lambda_i}\beta_{X_i})+p_i^*(\alpha-\frac{1}{\lambda_i}\beta_{X_i})-(\alpha-\frac{1}{\lambda_i})m_i E_i\] 
 \[=(p_i)^*(K_{X_i}+B_{X_i}+\alpha\beta_{X_i})-(\alpha-\frac{1}{\lambda_i})m_iE_i\]
 where $(\alpha-\frac{1}{\lambda_i})m_i>0$. Therefore the $\alpha$-discrepancies are improved on $Z_i$ and hence, $(Z_i, B_{Z_i}+\alpha\beta_{Z_i})$ is also a glc pair.\\
 
Every link $X_i\rightarrow X_{i+1}$ is an outcome of $K+B+\frac{1}{\lambda}\beta$-MMP over $S_i$ after taking a maximal divisorial blowup $Z_i\overset{p_i}{\rightarrow} X_i $, consisting of a finite number of $K+B+\frac{1}{\lambda}\beta$-flips  \[
Z_i=Z_i^0\dashrightarrow Z_i^1...\dashrightarrow Z_i^m
\]
 followed by a divisorial contraction $q_i^m:Z_i^m \rightarrow X_i^{m+1}=X_{i+1}$.\\
 Now we claim that the birational map $Z_i^k\dashrightarrow Z_i^{k+1}$ over $X_i^{k+1}$ is a step of $K+B+\alpha \beta$-MMP. We prove this by induction on $k$. Firstly we have that $\alpha>\frac{1}{\lambda}$ and $K_{Z_i}+B_{Z_i}+\alpha \beta_{Z_i}=p_i^0(K_{X_i}+B_{X_i}+\alpha \beta_{X_i})-aE_i$ where $a>0$ as $p_i^0$ is contraction of a negative extremal ray. Therefore we have that $(K_{Z_i^0}+B_{Z_i^0} +\alpha \beta_{Z_i^0})\cdot P_i^0>0$ where $P_i^0$ is the negative extremal ray contracted by $p_i^0$.\\
 
Note that $\rho(Z_i/S_i)=2$. Thus the other extremal ray $Q_i^0$ of $\NE(Z_i)=\NE(Z^0/S_i)$ satisfies $(K_{Z_i^0}+B_{Z_i^0}+\alpha \beta_{Z_i^0})\cdot Q_i^0 <0$ as $(K_{Z_i^0}+B_{Z_i^0}+\alpha \beta_{Z_i^0})$ is not relatively nef over $S_i$.
 By induction, assume that we have $(K_{Z_i^k}+B_{Z_i^k}+\alpha \beta_{Z_i^k})\cdot P_i^k>0$, where $P_i^k$ is the extremal ray contracted by $p_i^k$. Let $Q_i^k$ be the other extremal ray in $\NE(Z_i^k/S_i)$. Note that since base $\dim S_i=2$, the general fibers of every map $Z_i^k \to S_i$ are isomorphic.
\[(K_{Z_i^k}+B_{Z_i^k}+\alpha\beta_{Z_i^k})\cdot F= (K_{X_i}+B_{X_i}+\alpha\beta_{X_i})\cdot F\]
Recall that $(K_{X_i}+B_{X_i}+\frac{1}{\mu}\beta_{X_i})\cdot F=0$. Thus it follows that
\[ (K_{X_i}+B_{X_i}+\alpha \beta_{X_i})\cdot F=(\alpha-\frac{1}{\mu})(-\mu(K_{X_i}+B_{X_i}))\cdot F\]
Note that if $\alpha < \frac{1}{\mu}$, then $K_{Z_i^k}+B_{Z_i^k}+\alpha \beta_{Z_i^k}$ is not nef over $S_i$ which implies that $(K_{Z_i^k}+B_{Z_i^k}+\alpha \beta_{Z_i^k})\cdot Q_i^k <0$. If $\alpha=\frac{1}{\mu}$, then $(K_{Z_i^k}+B_{Z_i^k}+\alpha \beta_{Z_i^k})\cdot F=0$. Suppose, if possible, that $(K_{Z_i^k}+B_{Z_i^k}+\alpha \beta_{Z_i^k})\cdot Q_i^k\geq 0$. Write $F=aP_i^k+bQ_i^k$ where $a,b\geq 0$. Since $(K_{Z_i^k}+B_{Z_i^k}+\alpha \beta_{Z_i^k})\cdot P_i^k >0, a$ must be $0$. $F=bQ_i^k$. Hence a general fiber $F$ lies on the ray $Q_i^k$. Hence it cannot be the next birational contraction. Thus $(K_{Z_i^k}+B_{Z_i^k}+\alpha \beta_{Z_i^k})\cdot Q_i^k <0$. 
  This shows that $Z_i^k\dashrightarrow Z_i^{k+1}$ over $X_i^{k+1}$ is a step of $K+B+\alpha \beta$-MMP. 
 
 Thus for every divisor $D$, $a(D, X_i,B_{X_i}+\alpha \beta_{X_i})\leq a(D, X_{i+1}, B_{X_{i+1}}+\alpha \beta_{X_{i+1}})$ and hence it proves that $\alpha$-discrepancy gets improved in each Sarkisov link.

\end{proof}

\begin{claim}\label{clm:alpha-gklt}
    The pair $(X_i,B_{X_i}+\alpha\beta_{X_i})$ is a generalized klt pair for sufficiently large $i$.
\end{claim}
\begin{proof}[Proof of claim \ref{clm:alpha-gklt}]
    By Claim \ref{clm:glc-i}, we can choose a natural number $K$ such that $(X_i, B_{X_i}+\alpha\beta_{X_i})$ is glc for all $i\geq K$. If possible, there are infinitely many $i$ such that the pair is not gklt. For each such $i$, we can choose a divisor $F_i$ over $X_i$ such that the discrepancy $a(F_i,X_i,B_{X_i}+\alpha\beta_{X_i})= -1$. By Claim \ref{clm: improve-dis}, 
    \[ -1\leq a(F_i,X_K,B_{X_K}+\alpha\beta_{X_K})\leq a(F_i,X_i,B_{X_i}+\alpha\beta_{X_i})=-1\]
    Let $(F_i)_{X_i}$ and $(F_i)_{X_K}$ be the centers of the divisor $F_i$ on $X_i$ and $X_K$ respectively. Let $x_i$
and $x_K$ be two general points in the centers. Then equality of discrepancy tells that the map $X_K\dashrightarrow X_i$ is an isomorphism in neighborhoods of $x_i$ and $x_K$. Thus the local $c_\theta$-generalized threshold must agree.
Since boundary coefficients are at most $\epsilon$, the discrepancy is at least $\geq-\epsilon>-1; F_i$ must be exceptional as the discrepancy of a boundary divisor is the same as the generalized discrepancy of that boundary divisor. 
Now, observe that $\alpha$-discrepancy of $F_i$ is $-1 < -\theta(F_i)$ as $\theta$ takes values in $[0,1)$. 
Recall that \[a(F_i, X_K, B_{X_K}+\alpha\beta_{X_K})=a(F_i, X_K, B_{X_K})- \alpha m_{F_i}\]
Note that $m_{F_i}>0$, otherwise $a(F_i,X_K,B_{X_K}+\alpha\beta_{X_K})=a(F_i,X_K,B_{X_K})=-1$ would imply $(X_K,B_{X_K})$ is not klt which is a contradiction.
Thus it follows that
\[a(F_i,X_K,B_{X_K})-\alpha m_{F_i}< -\theta(F_i)\]
\[\frac{a(F_i, X_K, B_{X_K})+\theta(F_i)}{m_{F_i}}< \alpha\]

Therefore, it follows that 
\[ \frac{1}{\lambda_i}\leq c_\theta(x_i; X_i)=c_\theta(x_K; X_K) < \alpha\]
But by Remark \ref{rmk:local threshold theta}, the local threshold on $X_K$ can take only finitely many values. Since $\frac{1}{\lambda_i}\to \alpha$, this gives a contradiction. Hence for sufficiently large $i$, $(X_i,B_{X_i}+\alpha\beta_{X_i})$ is a generalized klt pair.\\
\end{proof}

Now choose a sufficiently large $K$ such that for all $i\geq K$, $(X_i, B_{X_i}+\alpha\beta_{X_i})$ is gklt.
$a(E, X_K,B_{X_K}+\alpha\beta_{X_K})\leq a(E, X_i,B_{X_i}+\alpha\beta_{X_i})$ as the $\alpha$-discrepancy gets improved.
Recall that the crepant divisor $E_i$ satisfies the following:
\[a(E_i, X_i,B_{X_i}+\alpha\beta_{X_1})=-\theta(E_i)-(\alpha-\frac{1}{\lambda_i})m_i< -\theta(E_i)\leq 0\]
Since $E_i$ survives as a divisor on $X_{i+1}$ by Claim \ref{clm:E survive}, the discrepancy would be $a(E_i,X_{i+1}, B_{X_{i+1}}+\alpha\beta_{X_{i+1}})=-\theta(E_i)$. 
This implies that all $E_i$ appearing in the above process are distinct. Note that all $E_i$ are exceptional over $X_K$. The process shows that there are infinitely many divisors $E$ such that $a(E, X_K, B_{X_K}+\alpha\beta_{X_K})<0$ since $(X_K, B_{X_K}+\alpha \beta_{X_K})$ is a generalized klt pair.
This gives a contradiction.
Thus the sequence has to terminate.

\end{proof}

\bibliographystyle{alpha}
\bibliography{bibliography}

\end{document}